\documentclass[a4paper,10pt]{article}
\usepackage{fullpage}
\usepackage[english]{babel}
\usepackage[latin1]{inputenc}
\usepackage{enumerate}
\usepackage{amstext}
\usepackage{epsfig}
\usepackage{amsmath}
\usepackage{amsfonts}
\usepackage{amssymb}
\usepackage{amsthm}
\usepackage{latexsym}
\usepackage{url}

\newcommand{\CC}{\mathbb{C}}
\newcommand{\TT}{\mathbb{T}}

\newcommand{\RR}{\mathbb{R}}
\newcommand{\NN}{\mathbb{N}}
\newcommand{\LL}{\mathcal{L}}
\newcommand{\CP}{\mathbb{CP}}
\newcommand{\ddbar}{\partial\overline{\partial}}

\newtheorem{theorem}{Theorem}[section]

\newtheorem{lemma}{Lemma}[section]

\newtheorem{definition}{Definition}[section]
\newtheorem{remark}{Remark}[section]
\begin{document}
\author{Carlos Pérez-Garrandés}                  
\title{Ergodicity of laminations with singularities in Kähler surfaces}
\date{July 2012}
\maketitle
\begin{abstract}
Let $\mathcal{F}$ be a holomorphic foliation on $\mathcal{M}$, a homogeneous compact Kähler surface, with only hyperbolic singularities. Let $\LL$ be a closed set saturated by leaves of the foliation, containing singularities and with every leaf dense on it. We say $\LL$ is a minimal lamination by Riemann surfaces with only hyperbolic singularities. If there are no positive closed currents directed by $\LL$, then there is a unique positive harmonic current directed by $\LL$ of mass one. This result was obtained previously for $\CP^2$ by Fornæss and Sibony, by developing intersection theory for currents. By applying the same theory we obtain the result for the rest of homogeneous compact Kähler surfaces.
\end{abstract}
\begin{section}{Introduction}
The aim of this work is to obtain a generalization for every compact homogenous Kähler surface of the main result obtained by Fornæss and Sibony in \cite{FS3} via the theory they developed in \cite{FS1} (see also \cite{FS2}). These works are devoted to prove the uniqueness of the ergodicity for laminations in $\CP^2$ without directed closed currents.\par
\begin{definition}
We say $(X,\LL,E)$ is a transversally Lipschitz lamination by Riemann surfaces with singular set $E\subset X$, if $X$ is a compact topological space  such that  for every $p\not\in E$ we can find local charts $\phi_i:\Delta\times \mathcal{T}\rightarrow X$ where $\Delta$ is the unit disk and $\mathcal{T}$ is a metric space. These charts satisfy that the change of coordinates is $\phi_i^{-1}\circ\phi_j(z,t)=(f_{ij}(z,t),h_{ij}(t))$ with $h_{ij}$ Lipschitz, $f_{ij}$ holomorphic in the first variable and Lipschitz in the second one. These local charts are called flow boxes.
\end{definition}

The laminations we will deal with in this article will be embedded in complex surfaces $\mathcal{M}$. Then, if $ \phi:\Delta_{\delta,\delta'}^2\rightarrow U\subset\mathcal{M}$ is a local chart from a polydisk of radii $\delta$ and $\delta'$ to  $\mathcal{M}$ centered at $p$, the plaques $\Gamma_w$ of these flow boxes can be written as graphs $(z,f_w(z))$ with $z\in \Delta_\delta$, $w\in \mathcal{T}\subset \Delta_\delta'$ and $f_w$ being an holomorphic function verifying that $f_w(0)=w$.\par
\begin{definition}
Let $(X,\LL,E)$ be a lamination by Riemann surfaces with singularities embedded on a compact complex surface $\mathcal{M}$, with $E$ discrete. We say that $p\in E$ is a hyperbolic singularity , if we can find $U\subset\mathcal{M}$ a neighbourhood of $p$ and some holomorphic coordinates $(z,w)$ centered on $p$ such that the leaves of $(X,\LL,E)$ are invariant varieties for the holomorphic 1-form $\omega=zdw-\lambda w dz$, with $\lambda\in\CC\setminus \RR$.
\end{definition}


\begin{theorem}\label{teoremaco}
Let $\mathcal{M}$ be a homogeneous compact Kähler surface containing a minimal Lipschitz lamination $\LL$ by Riemann surfaces with hyperbolic singularities. If there are no closed currents directed by $\LL$, then there is a unique directed harmonic current of mass one.
\end{theorem}

This theorem was proven by Fornæss and Sibony in \cite{FS3} for $\CP^2$. following Tits  \cite{Ti} there are only three other cases of homogeneous compact Kähler surfaces: $\CP^1\times\CP^1, \CP^1\times\TT^1, \TT^2$. The theory developed in \cite{FS1} works for every compact homogenous Kähler manifold, and according to that paper we just need to prove that the geometric selfintersection of a harmonic current always vanishes. \par

The reasoning in order to prove the theorem will be similar in the three cases. The proof will be made explicitly for $\CP^1\times\CP^1$. 
\end{section}
\begin{section}{Directed harmonic currents in laminations}
Let $(\mathcal{M},\omega)$ be a compact homogeneous Kähler surface. The space of $(1,1)$-forms can be endowed with the supremum norm, in this way the space of $(1,1)$-forms gets a structure of Banach space. A $(1,1)$-current $T$ of order 0 is a $\CC$ linear functional $T$ on this space. We will deal with positive harmonic currents directed by a lamination. It means that $T(\phi)\geq 0$ if $\phi$ is a positive $(1,1)$-form, $T_{|U}(\gamma)=0$ where $\gamma$ is a $(1,0)$-form in $U$ holding that the plaques of the lamination in $U$ are integral varieties of $\gamma$ and $T(\ddbar u)=0$ for every real function $u$. More information about currents can be found in Demailly's book \cite{De}.\par
This kind of currents can be decomposed in regular flow boxes as
$$T=\int_{t\in \mathcal{T}} h_t[\Gamma_t]d\mu(t).$$
where $h_t$ is a harmonic function along each plaque depending on the point in the transversal $t\in\mathcal{T}$, $[\Gamma_t]$ is the integration current over the plaque $\Gamma_t$ and $\mu$ is a local transversal measure.\par
\begin{remark}[\cite{FS1}] If a positive harmonic current on a laminated compact set $X$ gives mass to a leaf, then this leaf is a compact Riemann surface. However, we will assume the non existence of closed leaves, then $\mu$ cannot have mass on points.
\end{remark}
We need to recall another important concept from \cite{FS1}. The geometric selfintersection of this kind of currents in homogeneous manifolds is a measure defined in a flow box as follows. Let $u$ be a test function with support in the flow box, then we define the geometric selfintersection as
$$T\wedge_g T(u)=\lim_{\epsilon\rightarrow 0}T\wedge_g T^\epsilon(u)= \int{\sum_{p\in J_{\alpha,\beta}^\epsilon} u(p)h_\alpha(p)h_\beta^\epsilon(p)d\mu(\alpha)d\mu(\beta)},$$
where we have chosen a family of automorphisms $\Phi_\epsilon\rightarrow Id$, such that $T^\epsilon:=\Phi_{\epsilon *}(T)$ and $\Gamma_\beta^\epsilon:=\Phi_\epsilon(\Gamma_\beta)$. The set $J^\epsilon_{\alpha,\beta}$ is the set of the intersection points between $\Gamma_\alpha$ and $\Gamma_\beta^\epsilon$.\par
\begin{theorem}[Fornæss, Sibony \cite{FS3}]
Let $\LL$ be a minimal Lipschitz lamination with only hyperbolic singularities in a compact homogeneous Kähler manifold $(\mathcal{M},\omega)$, without directed closed currents. If $T\wedge_g T=0$ for every directed positive harmonic current $T$ there is only one of them with mass one.
\end{theorem}
The existence of such currents was proven for laminations by Riemann surfaces with finite number of singularities by Berndtsson and Sibony in \cite{BS}.\par
Proving that the geometric selfintersection is equal to $0$ in a regular flow box is done by finding a bound $N$ on the number of intersection points of a plaque of the lamination and another plaque moved by an element of this family of automorphisms close enough to the identity. Indeed, if $u$ is a continuous function which is $0$ outside a flow box, since $u$ and $h_\alpha$ are bounded, then
$$|T\wedge_g T^\epsilon(u)|\leq \int_{d(\Gamma_\alpha,\Gamma_\beta)<C\epsilon}{N|\sup (h_\alpha)|^2|\sup u|d\mu(\alpha)d\mu(\beta)}\rightarrow 0$$
where $C>0$ is a constant. This limit holds because $\mu$ has no mass on single points.\par
Under some assumptions, that we recall at the beginning of section 3, in \cite{FS3} it is also proven that the geometric self-intersection is $0$ in an open linearizable neighbourhood of a hyperbolic singularity.\par
To sum up, we need to find a continuous family of automorphisms of each surface $\Phi_\epsilon$ with $\Phi_0= Id$ satisfying the conditions of \cite{FS3} around the singularities and a big enough $N\in\NN$ to ensure that $\Gamma_1$ and $\Gamma_2^\epsilon$ intersect each other at most at $N$ points when $\epsilon$ is small enough. $\Gamma_1$ and $\Gamma_2$ denote plaques of the same flow box, and  $\Gamma_2^\epsilon= \Phi_\epsilon(\Gamma_2)$. Hence, in order to prove Theorem \ref{teoremaco}, we just need to prove, for each compact complex Kähler surface $\mathcal{M}$ a theorem like the one below which is the key result in this article.
\begin{theorem} \label{teoprinc}
Let $\LL$ be a minimal transversally Lipschitz lamination with only hyperbolic singularities in $\mathcal{M}$ and without directed closed currents. There are a covering of the lamination outside the singular neighbourhoods, a $N_\LL\in \NN$ and a $\epsilon_\LL>0$ such that there are at most $N$ intersections between $\Gamma_1$ and $\Gamma_2^\epsilon$, where  $\Gamma_1$ and $\Gamma_2$ are plaques in the same flow box, for every $\epsilon$ with $|\epsilon|<\epsilon_\LL$.
\end{theorem}
The proof relies in finding a covering by flow boxes of three different types according to the behaviour of the lamination and to the family of automorphisms inside them. Firstly, open linearizable neighbourhoods around the singularities where results of \cite{FS3} can be applied. 
Secondly, some flow boxes having intersection points between a plaque and itself moved by an element of the family of automorphims close enough to the identity. And finally, flow boxes where, if $\Gamma_1$ and $\Gamma_2^\epsilon$ are very close, then $\Gamma_1$ and $\Gamma_2$ are far away. The final step is to get a contradiction when there are too many intersection points of a original plaque and a perturbed one. It is done by following a leaf up to one of the plaques of the second type and showing that there is a plaque which does not intersect itself moved by the automorphism.\par
Since different surfaces have different group of automorphisms we cannot consider the same family for these three cases, but we will search for automorphisms with similar local behaviour. Main difference with the case of $\CP^2$ is that the automorphisms taken in that case by Fornæss and Sibony have a line of fixed points. In our cases there is at most one fixed point.\par
In order to find these coverings, Hurwitz's Theorem will play a special role. Since we can see plaques locally as graphs, and these graphs vary continuously in the transversal, we can apply Hurwitz's Theorem. It is useful to prove that, when we move a plaque by a tangential motion, the moved plaque and the original one intersect each other.\par
We will also need the following remark.
\begin{remark}\label{remark1}Writing the plaques of the flow boxes as graphs, Lipschitzness implies that for $t, t'\in \mathcal T\subset \Delta_{\delta'}$, there is a constant $C>1$ depending on the flow box, such that 
$$\frac{d(t,t')}{C}\leq d(f_t(z),f_{t'}(z))\leq Cd(t,t')$$
for every $z\in \Delta_\delta$. In the expression above $d$ is the distance in the transversal. We will denote $d(f_t(z),f_{t'}(z))$ by $d_z(\Gamma_t,\Gamma_{t'})$.
\end{remark}
This remark will be necessary to travel from flow box to flow box and we will use it sistematically in our arguments. \par
\end{section}
\begin{section}{Case of $\CP^1\times\CP^1$}
We consider $\CP^1\times\CP^1$ with the Fubini-Study metric in each factor. Since it is a product space then $T(\CP^1\times\CP^1)=T\CP^1\times T\CP^1$. Hence, we have a notion of verticality and horizontality in the tangent bundle defined in the natural way.\par
Assume that the lines $[1:0]\times\CP^1$ and $\CP^1\times[1:0]$ do not contain any singularity, $p=([1:0],[1:0])\in\LL$ and $T_p\LL$ is not vertical neither horizontal.\par
Therefore, we have four different charts covering $\CP^1\times\CP^1$, $\psi_i:\CC^2\rightarrow\CP^1\times\CP^1$ for $i=1,2,3,4$ defined as follows:
\begin{enumerate}[a)]
\item $\psi_1(z,w)=([z:1],[w:1]),$
\item $\psi_2(z,w)=([1:z],[w:1]),$
\item $\psi_3(z,w)=([z:1],[1:w]),$
\item $\psi_4(z,w)=([1:z],[1:w]).$
\end{enumerate}
Clearly every singularity is contained in the image of $\psi_1$.\par
The family of automorphisms we are searching for is 
$$\Phi_\epsilon([z_1:z_2],[w_1:w_2])=([z_1+\epsilon v_1 z_2:z_2],[w_1+\epsilon v_2 w_2:w_2])$$ 
for a suitable vector $(v_1,v_2)$, but we have to choose it carefully according to the behaviour of the lamination in a neighbourhood of a singularity.\par
Let $s_1,s_2,\dots,s_n$ be the singularities. Since they are hyperbolic, there exist $A_{i_A}$ a linearizable neighbourhood around $\psi_1^{-1}(s_{i_A})$ and a change of coordinates $\phi_{i_A}:A_{i_A}\rightarrow \Delta_{\delta,\delta'}^2$ with $\phi_{i_A}(\psi^{-1}(s_{i_A}))=(0,0)$ such that in the new coordinates $(z',w')$ the leaves of the lamination are integral varieties of the $1$-form $w'dz'-\lambda_{i_A} z' dw'$ with this $\lambda_{i_A}$ veryfing that $\lambda_{i_A}\not\in\RR$. Hence the separatrices are $\{w'=0\}$ and $\{z'=0\}$. $\Phi_\epsilon$ would act as a translation by $(\epsilon v_1,\epsilon v_2)$ in $\psi_1^{-1}(A_{i_A})=\Delta_{\delta,\delta'}^2$, $\Phi_\epsilon(z,w)=(z+\epsilon v_1,w+\epsilon v_2)$.\par
Next we define $\Phi_\epsilon^{i_A}=\phi_{i_A}^{-1}\Phi_\epsilon\phi_{i_A}$, and $\Phi_\epsilon^{i_A}$ has to hold the conditions of \cite{FS3}: it could be written as $(\alpha(\epsilon),\beta(\epsilon))+(z',w')+\epsilon O(z',w')$ with $\alpha'(0),\beta'(0)\not= 0$ and $\frac{\beta'(0)}{\alpha'(0)}\not=\lambda_{i_A}$. Notice that $(\alpha'(0),\beta'(0))=D\phi_{{i_A}|\Phi_\epsilon(0,0)}^{-1}(v_1,v_2)=:(v_1^{i_A},v_2^{i_A})$. The third element of the sum appears if and only if $\phi$ is not linear. In fact, it is not linear because in that case the lamination would have a directed closed current, the integration current on the separatrix, which would be a projective line. These conditions must hold around every singularity. Therefore we have to choose a vector $(v_1,v_2)$ such that:
\begin{enumerate}[i)]
\item $v_1^{i_A}, v_2^{i_A}\not=0$ and $\frac{v_2^{i_A}}{v_1^{i_A}}\not=\lambda_{i_A}$,\label{cond1}
\item $(v_1,v_2)$ is unitary,\label{cond2}
\item $v_1,v_2\not=0$ and $(v_2,v_1)$ does not belong to $T_p\LL$,
\item $(v_1,v_2)$ is tangent for certain point $p'\in\CC^2\setminus(\bigcup A_{i_A})$\label{cond4}.
\end{enumerate}
So, we have fixed $(v_1,v_2)$ and we have the family of automorphisms $\Phi_\epsilon$. The next step is choosing a good covering of the lamination $\LL$ as follows:

\begin{enumerate}[(1)]
\item We already have linearizable neighbourhoods of the singularities where \cite{FS3} can be applied, we will denote them by $A_{i_A}$. We will call them singular neighbourhoods.
\item We need a neighbourhood $U_0$ of $p$, because it is a fixed point for every element of the family of automorphisms. We will find it by using $\psi_4$. 
\item Afterwards, we will cover $\CP^1\times[1:0]\setminus U_0$ via $\psi_3$ with two types of flow boxes, horizontal $W_{j_W}^a$ and and vertical $W_{i_W}^t$. The superindices come from ``along'' and ``transversal'', referring to the behaviour of the laminations with respect to the automorphisms.
\item Same for $[1:0]\times\CP^1\setminus U_0$ with $\psi_2$. We will obtain $V_{i_V}^t$ and $V_{j_V}^a$. 
\item And finally, by using $\psi_1$, we get flow boxes $B_{j_B}^a$ and $B_{i_B}^t$ covering the rest of the points of $\CP^1\times\CP^1$ depending on whether every plaque is transversal to the motion or not, respectively.
\end{enumerate}
\begin{lemma}
There is a flow box $U_0$ centered at $p=([1:0],[1:0])$ biholomorphic to $\Delta_\delta\times\mathcal{T}$  and an $\epsilon_0>0$ such that, if $\Gamma_w$ and $\Gamma_{w'}^\epsilon$ intersect each other in $N_0$ points, then the vertical distance in $|z|=\delta$ verifies
$$d_z(\Gamma_w,\Gamma_{w'})> c_0 |\epsilon|$$
with certain $c_0>0$ for every $\epsilon$ with $|\epsilon|<\epsilon_0$.
\end{lemma}
\begin{proof}
We will use $\psi_4$. Consider a horizontal flow box $U_0'=\Delta_\delta\times \mathcal{T}$ centered at $p$, $\Delta_\delta$ is a disk centered at $0$, and $\mathcal{T}$ is a topological space containing $0$. The points in the flow box can be written as $(z,w+f_w(z))$, where $f_w$  are holomorphic functions satisfying $f_w(0)=0$ for every $w\in\mathcal{T}$.\par
Since $f_0'(0)\not=0$ and $(v_2,v_1)$ is not a scalar multiple of $(1,f_0'(0))$, we can choose $U_0$ verifying that $m<|f_w'(z)|<M$, $|f_w'(z)-\frac{v_1}{v_2}|>m_0>0$ for every $(z,w)\in\Delta_\delta\times\mathcal{T}$ and as $f_w(z)=g_w(z)z$ for a certain holomorphic function $g_w$ varying continuously with $w$, we can also ask for $m<|g_w(z)|<M$ and $|g_w(z)-\frac{v_1}{v_2}|>m_0>0$ for every $(z,w)\in \Delta_\delta\times\mathcal{T}$.\par
Now, we want to find $\delta_0$ small enough to get that if $\Gamma_w$ and $\Gamma_{w'}^\epsilon$ intersect each other in $N_0$ points, then the vertical distance in $z$ satisfies 
$$d_z(\Gamma_w,\Gamma_{w'})>d_z(\Gamma_{w'},\Gamma_{w'}^\epsilon)-d_z(\Gamma_w,\Gamma_{w'}^\epsilon)> c_0 |\epsilon|$$
with certain $c>0$ for every $z$ with $|z|=\delta_0$.
The idea is to find a lower bound for $d_z$. Since $\LL$ is a Lipschitz lamination, we can find the bound for $\Gamma_0$ and shrink later the transversal to ensure that every plaque holds the inequality.\par
In the domain of $\psi_4$, 
$$\Phi_\epsilon(z,w)=\left(\frac{z}{1+\epsilon v_1 z},\frac{w}{1+\epsilon v_2 w}\right),$$
then
$$\Gamma_0^\epsilon=\left\{\left(\frac{z}{1+\epsilon v_1 z},\frac{f_0(z)}{1+\epsilon v_2 f_0(z)}\right),z\in \Delta_\delta\right\}.$$
Hence, if we fix $z\in \Delta_\delta$ such that $z'=\frac{z}{1+\epsilon v_1 z}\in \Delta_\delta$, then $z=\frac{z'}{1-\epsilon v_1 z'}$. So, the transversal distance in a point $z$ is 
$$d_z(\Gamma_0,\Gamma_0^\epsilon)=\left|f_0(z)-\frac{f_0(\frac{z}{1-\epsilon v_1 z})}{1+\epsilon v_2 f_0(\frac{z}{1-\epsilon v_1 z})}\right|.$$
We can write it as follows
\begin{align*}
d_z(\Gamma_0,\Gamma_0^\epsilon)=&\left|zg_0(z)-\frac{(\frac{z}{1-\epsilon v_1 z})g_0(\frac{z}{1-\epsilon v_1z})}{1+ \frac{z\epsilon v_2}{1-\epsilon v_1z}g_0(\frac{z}{1-\epsilon v_1 z})}\right|\\
=&\left|zg_0(z)-\frac{zg_0(\frac{z}{1-\epsilon v_1 z})}{1+ z\epsilon \left(-v_1+v_2g_0(\frac{z}{1-\epsilon v_1 z})\right)}\right|\\
=&\left|\frac{z\left[g_0(z)-\epsilon z g_0(z)\left(v_2g_0\left(\frac{z}{1-\epsilon v_1 z}\right)-v_1\right)-g_0\left(\frac{z}{1-\epsilon v_1 z}\right)\right]}{1+\epsilon z\left(-v_1+v_2g_0\left(\frac{z}{1-\epsilon v_1 z}\right)\right)}\right|\\
\geq &\left| \frac{z}{1+z\epsilon\left(-v_1+v_2 g_0\left(\frac{z}{1-\epsilon v_1 z}\right)\right)}\right|\left( F-G\right) ,
\end{align*}
where
\begin{align*}
F:=&\left|\epsilon z g_0(z)\left(v_2g_0\left(\frac{z}{1-\epsilon v_1 z}\right)-v_1\right)\right|,\\
G:=&\left| g_0(z)-g\left(\frac{z}{1-\epsilon v_1 z}\right)\right|.
\end {align*}
We are searching for a lower bound of this last expression. $F$ is obviously greater than $|\epsilon||z| m m_0 |v_2|$ so we have to find an upper bound for $G$. We observe that $\frac{z}{1-\epsilon v_1 z}=z+\frac{\epsilon v_1 z^2}{1-\epsilon v_1 z}$, and considering Taylor expansion of $g_0$ in $0$, we obtain that
\begin{align*}
\left| g_0(z)-g\left(\frac{z}{1-\epsilon v_1 z}\right)\right|&=\left|\sum_{n=p}^\infty a_n z^n-\sum_{n=p}^\infty a_n \left(z+\frac{\epsilon v_1 z^2}{1-\epsilon v_1 z}\right)^n\right|\\&=|\epsilon v_1z^{p+1}h_\epsilon(z)|,
\end{align*}
with $|h_\epsilon(z)|$ bounded by $M_0$ for every $z$ and every $\epsilon$ small enough.\par
Thus, by replacing these bounds in the previous expression,
\begin{align*}
d_z(\Gamma_0,\Gamma_0^\epsilon)&\geq\frac{|z|}{1+z\epsilon\left(-v_1+v_2 g_0\left(\frac{z}{1-\epsilon v_1 z}\right)\right)}\left[|\epsilon||z|m|v_2|m_0-|\epsilon v_1 z^{p+1} h(z)|\right]\\
&\geq \frac{|\epsilon z^2|}{1+z\epsilon\left(-v_1+v_2 g_0\left(\frac{z}{1-\epsilon v_1 z}\right)\right)}(mm_0|v_2|-v_1|z|^pM_0).
\end{align*}
Now, we choose $\epsilon_0$ such that if $|\epsilon|<\epsilon_0$ then 
$$\frac{1}{1+z\epsilon\left(-v_1+v_2 g_0\left(\frac{z}{1-\epsilon v_1 z}\right)\right)}>\frac{1}{2},$$
for every $z\in\Delta_\delta$, and if we set $\delta$ to satisfy that $mm_0|v_0|>2|v_1|\delta^p M_0$, then
$$\min_{|z|=\delta}d_z(\Gamma_0,\Gamma_0^\epsilon)>\frac{\delta^2|\epsilon|mM_0|v_2|}{4}.$$
Therefore $\min_{|z|=\delta}d_z(\Gamma_w,\Gamma_{w'})\geq \min_{|z|=\delta}d_z(\Gamma_{w'},\Gamma_{w'}^\epsilon)-\max_{|z|=\delta}d_z(\Gamma_w,\Gamma_{w'}^\epsilon)$
then, by applying Lemma 4.2 of \cite{FS1} $$\min_{|z|=\delta}d_z(\Gamma_w,\Gamma_{w'})\geq\frac{\delta^2|\epsilon|mM_0|v_2|}{4}-c^N_0K|\epsilon|.$$ Hence if $N_0$ is big enough,  $$\min_{|z|=\delta}d_z(\Gamma_w,\Gamma_{w'})\geq\frac{\delta^2|\epsilon|mM_0|v_2|}{8}>0.$$
So the number $c_0$ we were searching for is $$c_0=\frac{\delta^2mM_0|v_2|}{8}.$$
\end{proof}

\begin{lemma}\label{lema1}
There is a covering of $\CP^1\times[1:0]\setminus U_0$ by flow boxes of two different types, $W_{j_W}^a$ and $W_{i_W}^t$ and an $\epsilon_1>0$,  
verifying that for every $\epsilon$ such that $|\epsilon|<\epsilon_1$,  
\begin{itemize}
\item if $\Gamma_w$ is a plaque in $W_{j_W}^a$ then $\Gamma_w^\epsilon\cap \Gamma_w \neq \emptyset$; 
\item if $\Gamma_z$ and $\Gamma_{z'}$ are plaques in $W_{i_W}^t$ satisfying that $\max d_w(\Gamma_z,\Gamma_{z'}^\epsilon)<\frac{|v_1||\epsilon|}{2}$ then $\min d_w(\Gamma_z,\Gamma_{z'})>\frac{|v_1||\epsilon|}{2}$.
\end{itemize}
\end{lemma}

\begin{proof}
In order to prove this lemma we use $\psi_3$. In this chart, an automorphism behaves as $\Phi_\epsilon(z,w)=(z+\epsilon v_1,\frac{w}{1+\epsilon v_1 w})$. It is a horizontal translation in $w=0$. We want to cover the points of $w=0$ which are not in $U_0$. It is a compact set, so we will find a finite covering.\par
If $q$ is a point with horizontal tangent, we take a horizontal flow box centered at $q$ where $f'_0(z)=0$ if and only if $z=0$. We will proof that for $\epsilon$ small enough, $\Gamma_0$ and $\Gamma_0^\epsilon$ intersect each other and by Hurwitz's theorem (see for example Conway's book \cite{Co}) we can find a flow box centered at $q$ verifying this for every plaque in it.\par
We can write $\Gamma_0=\{(z,f_0(z)),\, z\in \Delta_{\delta'}\}$ with $f_0(0)=0$ and $f'(0)=0$ and $\Gamma_0^\epsilon=\{(z+\epsilon v_1,\frac{f_0(z)}{1+\epsilon v_2f_0(z)}),\, z\in \Delta_{\delta'}\}$, so we want to compute if the function 
$$f_0(z)-\frac{f_0(z-\epsilon v_1)}{1+\epsilon v_2 f_0(z-\epsilon v_1)}$$ has any zero. The number of zeros of that function is the same as the number of zeros of 
\begin{align*}
g_\epsilon(z)=&\frac{1}{\epsilon} \left( f_0(z)-\frac{f_0(z-\epsilon v_1)}{1+\epsilon v_2 f_0(z-\epsilon v_1)} \right)\\
=&\frac{1}{\epsilon}\left( f_0(z)-f_0(z-\epsilon v_1)-\frac{f_0^2(z-\epsilon v_1)\epsilon v_2}{1+\epsilon v_2 f_0(z-\epsilon v_1)} \right).
\end{align*}
Then, $\lim_{\epsilon\rightarrow 0}g_\epsilon(z)=f'_0(z)v_1-f_0^2(z)v_2$ which has a finite number of zeroes in $\Delta_\delta$. By Hurwitz's theorem again, there is $\epsilon_1$ such that if $|\epsilon|<\epsilon_1$, $g_\epsilon(z)$ has the same number of zeros as the limit, hence $\Gamma_0^\epsilon$ and $\Gamma_0$ intersect each other. So do nearby enough plaques. We cover these points by flow boxes $W_{j_W}^a$.\par

Now, if $q$ is a non horizontal point in $w=0$, we can take a vertical flow box around it $(z+f_z(w),w)$ and $\Gamma_z^\epsilon=(z+\epsilon v_1+f_z(w),\frac{w}{1+\epsilon v_2 w})$.
If $\max d_w(\Gamma_z,\Gamma_{z'}^\epsilon)<|v_1\epsilon|/2$, then
$$\min d_w(\Gamma_{z},\Gamma_{z'})\geq \min d_w(\Gamma_{z},\Gamma_{z'}^\epsilon)-\max d_w(\Gamma_{z'},\Gamma_{z'}^\epsilon)=|\epsilon v_1|-|v_1\epsilon|/2>|\epsilon v_1|/2.$$
In this way we obtain the flow boxes $W_{i_W}^t$.\par
So, finally, we can cover $\{w=0\}\setminus U_0$ by a finite number of flow boxes.
\end{proof}

We can cover $[1:0]\times\CP^1$ analogously and obtain the same result for open sets $V_{i_V}^t$ and $V_{j_V}^a$.
 
 \begin{lemma}\label{lema2}
There is a covering of $[1:0]\times\CP^1\setminus U_0$ by flow boxes of two different types, $V_{j_V}^a$ and $V_{i_V}^t$ and an $\epsilon_2>0$, 
verifying that for every $\epsilon$ such that $|\epsilon|<\epsilon_2$,  
\begin{itemize}
\item if $\Gamma_z$ is a plaque in $V_{j_V}^a$ then $\Gamma_z^\epsilon\cap \Gamma_z \neq \emptyset$; 
\item if $\Gamma_{w}$ and $\Gamma_{w'}$ are plaques in $V_{i_V}^t$ satisfying that $\max d_z(\Gamma_{w},\Gamma_{w'}^\epsilon)<\frac{|v_2||\epsilon|}{2}$ then $\min d_z(\Gamma_{w},\Gamma_{w'})>\frac{|v_2||\epsilon|}{2}$.
\end{itemize}
\end{lemma}

Define $W:=\bigcup (W_{j_W}^a)\cup \bigcup (W_{i_W}^t)$ and $V:=\bigcup (V_{j_V}^a)\cup \bigcup (V_{i_V}^t)$.
\begin{lemma}\label{lema3}
There is a covering of $\CP^1\times\CP^1\setminus (U_0\cup V\cup W\cup A)$ by flow boxes of two different types, $B_{j_B}^a$ and $B_{i_B}^t$, and an $\epsilon_3>0$ such that if $|\epsilon| <\epsilon_3$,
\begin{itemize}
\item if $\Gamma_w$ is a plaque in $B_{j_B}^a$ then $\Gamma_w^\epsilon\cap\Gamma_w\neq\emptyset$;
\item if $\Gamma_z$ and $\Gamma_{z'}$ are plaques in $B_{i_B}^t$ satisfying $\max d_w(\Gamma_z,\Gamma_{z'}^\epsilon)<\frac{|\epsilon|}{2}$ then $\min d_w(\Gamma_z,\Gamma_{z'})>\frac{|\epsilon|}{2}$
\end{itemize}
\end{lemma}
\begin{proof}
We use $\psi_1$ because every point of $\CP^1\times\CP^1\setminus(U_0\cup W \cup V\cup A)$ is on its domain. In this chart, $\Phi_\epsilon$ works as a translation by the vector $(\epsilon v_1,\epsilon v_2)$, and there is a point on this open set $p'$ whose tangent space contains $(v_1,v_2)$.

We will do a change of coordinates just for simplicity. Let us consider the rotation $R:\CC^2\rightarrow\CC^2$ sending $(v_1,v_2)$ to $(1,0)$ and $(v_2,-v_1)$ to $(0,1)$. We have obtained new coordinates $(z',w')$ such that our family of automorphisms is a family of horizontal translations. Then, we can argue as we did in \cite{PG}. The reader can check the explicit computations there, but we include here an overview of the arguments for the convenience of the reader. Let $p_h$ be a point where the motion is tangent to its plaque at it. Then, applying Hurwitz's theorem in the same way, this plaque moved a little bit by the family of automorphisms intersects the original plaque. We cover this kind of points with flow boxes $B_{j_B}^a$. The rest of the points are transversal to the motions, hence they can be covered with flow boxes $B_{i_B}^t$.\par 
The estimates appearing in the statement for $B_{i_B}^t$ follow from Remark \ref{remark1} and the fact that $d_w(\Gamma_z,\Gamma_z^\epsilon)=\epsilon$. This finishises the proof of the lemma. 
\end{proof}

Although we have several types of flow boxes covering the lamination in $\CP^1\times\CP^1$, we can split them in three main types: flow boxes along the automorphisms, which are $W_{j_W}^a,V_{j_V}^a,B_{j_B}^a$, transversal to the automorphisms $W_{i_W}^t,V_{i_V}^t,$ $B_{i_B}^t,$ $U_0$ and a singular flow box $A_{i_A}$ for each singularity. We set $\epsilon_4=\min_{i=0\dots 3}\{\epsilon_i\}$ and $c_4=\min\{c_0,|v_1|/2,|v_2|/2,1/2\}$. Now we are ready to prove Theorem \ref{teoprinc} for $\mathcal{M}=\CP^1\times\CP^1$.

\begin{theorem} \label{teo1}Let $\LL$ be a minimal transversally Lipschitz lamination with only hyperbolic singularities in $\CP^1\times\CP^1$ and without directed closed currents. There are a covering of the lamination outside the singular neighbourhoods and an $N_0$ such that there are at most $N_0$ intersections between $\Gamma_1$ and $\Gamma_2^\epsilon$, where  $\Gamma_1$ and $\Gamma_2$ are plaques in the same flow box, for every $\epsilon$ with $|\epsilon|<\epsilon_4$.\end{theorem}
\begin{proof}

For the sake of simplicity, throughout the proof we will denote by $d_{{max}}(\Gamma_1,\Gamma_2)$ the maximum of the transversal distances in a flow box between the plaques $\Gamma_1,\Gamma_2$ and $d_{{min}}(\Gamma_1,\Gamma_2)$ to the minimum.\par
By Lemma 4.2.a of \cite{FS1}, if $\Gamma_1$ $\Gamma^\epsilon_2$ are plaques in the same regular flow box which intersect each other in $N$ points, then the transversal distance satisfies that $d_{{max}}(\Gamma_1,\Gamma^\epsilon_2)<c^N|\epsilon|A$, for certain constants $c<1$ and $A>0$ not depending on the flow box. There exists $b>0$ such that the distortion of the transversal distance in a change of flow boxes is bounded from above by $b$ and by $1/b$ from below. This $b$ arises from combining the constant in Remark \ref{remark1} and the distortion of the distance when we change coordinates on the surface. Finally, there is also $M\in\NN$ holding that, for every plaque in a flow box along the motion, we can find a path from this plaque to a plaque in a flow box transversal to the motion passing through at most $M$ changes of flow boxes avoiding $A_{i_A}$ and $U_0$ (unless we have started in $U_0$). This number $M$ can also be chosen holding the same statement when starting from a flow box transversal to the motion and finishing in a tangential one.\par
Now, take $\Gamma_1$ and $\Gamma_2$ in a flow box transversal to the motion holding that $\Gamma_1$ and $\Gamma^\epsilon_2$ have $N$ intersection points for an $\epsilon$ with $|\epsilon|<\epsilon_4$. Hence $d_{{max}}(\Gamma_1,\Gamma^\epsilon_2)<c^NA|\epsilon|$. Consider a path as we said before joining this flow box transversal to the motion with another one along the motion, and let $\Gamma'_1$ and $\Gamma'_2$ be the corresponding continuation of the plaques. Then, by applying Lemma 4.2.b of \cite{FS1} when changing flow boxes, $d_{{max}}(\Gamma'_1,\Gamma^{\epsilon'}_2)<b^Mc^{N/2^M}|\epsilon|A$. Nevertheless, if $c^{N} A<c_4$ by the previous lemmas  $d_{{min}}(\Gamma_1,\Gamma_2)>c_4|\epsilon|$. Following the path we can also conclude that $d_{{min}}(\Gamma_1,\Gamma_2)>\frac{|\epsilon|c_4}{b^M}$. Then, $$d_{{min}}(\Gamma'_1,\Gamma^{\epsilon'}_1)>d_{{min}}(\Gamma'_1,\Gamma'_2)-d_{{max}}(\Gamma'_1,\Gamma^{\epsilon'}_2)\geq|\epsilon|\left(\frac{c_4}{b^M}-b^Mc^{N/2^M}A\right)$$
There is $N_1\in \NN$ such that if $N>N_1$, this last term is bigger than zero, but if this happens, it would mean that $\Gamma'_1$ and $\Gamma^{\epsilon'}_1$ do not have a common point. But they do if $|\epsilon|<\epsilon_4$. So $N$ cannot be arbitrarily large.

Now, we argue when we start in a flow box along the motion. Consider $\Gamma_1$ and $\Gamma_2$ in it such that $\Gamma_1$ and $\Gamma_2^\epsilon$ intersect each other at $N$ points. They also verify that $d_{{max}}(\Gamma_1,\Gamma^{\epsilon}_2)<c^N|\epsilon|A$. We construct a path to a transversal flow box, and we reach the continuation of the plaques $\Gamma_1'$ and $\Gamma_2'$. They hold that $d_{{max}}(\Gamma_1',\Gamma^{\epsilon'}_2)<Ab^Mc^{N/2^M}|\epsilon|$. Hence, there exists $N_2'\in\NN$ such that, if $N>N_2'$, then $c^{N/2^M} A b^M<c_4$. Therefore, by previous lemmas, $d_{min}(\Gamma_1',\Gamma'_2)>c_4 |\epsilon|$. We follow the path back to the original flow box and we get that
$d_{min}(\Gamma_2,\Gamma_2^\epsilon)>(c_4/b^M-c^N A)|\epsilon|$. So there is $N_2>N_2'$ holding that $c_4/b^M-c^N A>0$ for every $N>N_2$. But this would mean that there are no intersection points between $\Gamma_2$ and $\Gamma_2^\epsilon$. Same contradiction arises.\par
In order to obtain the $N_0$ in the statement, take $N_0=\max\{N_1,N_2\}$.\end{proof}
\end{section}
\begin{section}{Case of $\TT^1\times\CP^1$ and $\TT^2$}
These four different local behaviours we saw in the previous section describe also every behaviour appearing in the two remaining surfaces to be studied. So we just need to put them in the right situation. Let us begin with $\TT^1\times\CP^1$.\par
Let $\Pi_1:\TT^1\times\CP^1\rightarrow \TT^1$ and $\Pi_2:\TT^1\times\CP^1\rightarrow \CP^1$ be the projections on each factor and $\pi:\CC\rightarrow\TT^1$ is the canonical projection in $\TT^1$. Let $s_1,\dots,s_n$ be the singularities of the lamination. We can find an automorphism of $\TT^1\times\CP^1$ such that $\TT^1\times [1:0]$ does not contain any singularity, and an open simply connected relatively compact set $U$ of $\CC$, which is a neighbourhood of a fundamental domain for the equivalence relation definig $\TT^1$, containing only one preimage by $\pi$ of the singularities.\par
In this case, we are going to search for a family of automorphisms as 
$$\Phi_\epsilon( [z],[w_1:w_2]) =([z+v_1\epsilon],[w_1+\epsilon v_2w_2: w_2]).$$
So, in the chart $\psi_2(z,w)=([z],[w:1])$ the automorphisms act as translations by a vector $(\epsilon v_1,\epsilon v_2)$. Thus, if we choose $(v_1,v_2)$ satisfing the conditions \ref{cond1}),\ref{cond2}) and \ref{cond4}) we asked for in the case of $\CP^1\times\CP^1$, we can argue in a similar way: firstly, we need to cover $\TT^1\times[1:0]$ in a special way and then, the rest of the points are a compact set in the other chart where the automorphisms act as translations, so we can cover it as we did for $\CP^1\times\CP^1$.
\begin{lemma}
There is a covering of $\TT^1\times[1:0]$ by flow boxes of two different types,  $V_{j_V}^a$ and $V_{i_V}^t$ and an $\epsilon_0>0$, holding that if $|\epsilon| <\epsilon_0$, 
\begin{itemize}
\item if $\Gamma_z$ is a plaque in $V_{j_V}^a$ then $\Gamma_z^\epsilon\cap\Gamma_z\neq\emptyset$;
\item if $\Gamma_w$ and $\Gamma_{w'}$ are plaques in $V_{i_V}^t$ satisfying that $\max d_w(\Gamma_z,\Gamma_{z'}^\epsilon)<\frac{|v_1||\epsilon|}{2}$ then $\min d_w(\Gamma_z,\Gamma_{z'})>\frac{|v_1||\epsilon|}{2}.$
\end{itemize}
\end{lemma}
\begin{proof}
We work with $\psi_1$. In this chart $\Phi_\epsilon(z,w)=(z+\epsilon v_1,\frac{w}{1+\epsilon v_1 w})$, hence it is a horizontal translation in $w=0$. Notice that this is the same situation we dealed with in Lemma \ref{lema2}, hence the proof is the same.\par
\end{proof}
We set $V=\bigcup V_{j_V}^a \cup \bigcup V_{i_V}^t$.
\begin{lemma}
There is a covering of $\TT^1\times\CP^1\setminus V$ by flow boxes of two different types, $B_{j_B}^a$ and $B_{i_B}^t$, and an $\epsilon_1>0$ such that if $|\epsilon| <\epsilon_1$
\begin{itemize}
\item if $\Gamma_w$ is a plaque in $B_{j_B}^a$ then $\Gamma_w^\epsilon\cap\Gamma_w\neq\emptyset$; 
\item if $\Gamma_z$ and $\Gamma_{z'}$ are plaques in $B_{i_B}^t$ satisfying that $\max d_w(\Gamma_z,\Gamma_{z'}^\epsilon)<\frac{|\epsilon|}{2}$ then $\min d_w(\Gamma_z,\Gamma_{z'})>\frac{|\epsilon|}{2}.$
\end{itemize}
\end{lemma}
The behaviour in the chart given by $\psi_2$ is a translation, so the proof is the same than Lemma \ref{lema3}. Setting $\epsilon_\LL=\min\{\epsilon_1,\epsilon_2\}$, both lemmas together let us prove Theorem \ref{teoprinc} for $\mathcal{M}=\CP^1\times\TT^1$ by the same reasoning of Theorem \ref{teo1}.

Finally, we deal with the case of $\TT^2$. Let $\Lambda$ be a lattice in $\CC^2$, and let $\pi:\CC^2\rightarrow\CC^2/\Lambda=\TT^2$ be the canonical projection. If $\LL$ is a minimal lamination with hyperbolic singularities embedded in $\TT^2$, we can consider a relatively compact simply connected open neighbourhood $U$ of $(0,0)$  in $\CC^2$ covering a fundamental domain of the equivalence relation defining $\TT^2$, and containing only one preimage of the singularities inside it and no one on its boundary. The family of automorphisms  we will consider is $\Phi_\epsilon[(z,w)]=[(z+\epsilon v_1,w+\epsilon v_2)]$, with $(v_1,v_2)$ chosen as before. $\Phi_\epsilon$ lifts to a translation $\tilde{\Phi}_\epsilon:\CC^2\rightarrow \CC^2$. We can argue as we did in Lemma \ref{lema3} and we get Theorem \ref{teoprinc} when $\mathcal{M}=\TT^2$ in the same way we proved Theorem \ref{teo1}.

\end{section}
\begin{section}*{Acknowledgements}
This article will be part of my Ph.D. thesis directed by John Erik Fornæss and Luis Giraldo. Most of it was carried out during a visit to Prof. Fornæss at the Norwegian University of Sciences and Technology (Trondheim, Norway). I would like to thank him for his enlightening comments and to the institution for its hospitality and financial support. It is also my pleasure to thank to Prof. Giraldo for his generous help, useful advices and constant encouragement. Finally, I also thank to the referee for pointing out some mistakes appearing in the previous version, as well as for several useful comments and suggestions that helped to improve the exposition.\par
Partially supported by Ministerio de Ciencia e Innovación (Spain), MTM2011-26674-C02-02. The author benefits from a predoctoral grant of Universidad Complutense de Madrid, which also partially supported the stay at Trondheim.
\end{section}
\bibliographystyle{plain}
\bibliography{mybib}
\begin{flushleft}
C. Pérez-Garrandés \\
Instituto de Matemáticas Interdisciplinar (IMI)\\
Departamento de Geometría y Topología\\
Facultad de Ciencias Matemáticas\\
Universidad Complutense de Madrid\\
              Plaza de las Ciencias 3\\
              28040 Madrid\\
           carperez@ucm.es\\
          \end{flushleft}
\end{document}